\documentclass[a4paper,12pt]{article}
\usepackage{amsmath,amsthm,amsfonts,amssymb,bbm}
\usepackage{graphicx}
\usepackage{hyperref}
\usepackage[hmargin=1in,vmargin=1in]{geometry}
\hypersetup{colorlinks=true}

\renewcommand{\d}{\mathrm d}

\newcommand{\R}{\mathbb R}
\newcommand{\Z}{\mathbb Z}
\newcommand{\wt}{\widetilde}

\renewcommand{\Re}{\operatorname{Re}}
\renewcommand{\Im}{\operatorname{Im}}
\newcommand{\Ai}{\operatorname{Ai}}

\newcommand{\C}{\mathcal{C}}
\newcommand{\D}{\mathcal{D}}

\newcommand{\id}{\mathbbm{1}}
\renewcommand{\O}{\mathcal{O}}
\renewcommand{\P}{\mathbf P}
\newcommand{\E}{\mathbf E}

\newcommand{\I}{\mathrm i}

\DeclareMathOperator*{\sgn}{sgn}

\newtheorem{proposition}{Proposition}[section]
\newtheorem{theorem}[proposition]{Theorem}
\newtheorem{lemma}[proposition]{Lemma}
\newtheorem{definition}[proposition]{Definition}

\theoremstyle{definition}
\newtheorem{remark}[proposition]{Remark}

\numberwithin{equation}{section}
\allowdisplaybreaks

\author{B\'alint Vet\H o\thanks{MTA--BME Stochastics Research Group, Egry J.\ u.\ 1, 1111 Budapest, Hungary. E-mail: {\tt vetob@math.bme.hu}}}

\title{Tracy--Widom limit of $q$-Hahn TASEP}
\date{}

\begin{document}

\maketitle
\sloppy

\begin{abstract}
We consider the $q$-Hahn TASEP which is a three-parameter family of discrete time interacting particle systems.
The particles jump to the right independently according to a certain $q$-Binomial distribution with parallel updates.
It is a generalization of the totally asymmetric simple exclusion process (TASEP) on $\Z$.
For step initial condition, we prove that the current fluctuation of $q$-Hahn TASEP at time $\tau$ is of order $\tau^{1/3}$ and asymptotically distributed as the GUE Tracy--Widom distribution.
We verify the KPZ scaling theory conjecture for the $q$-Hahn TASEP.


\noindent {\bf MSC classes:} 60K35, 60B20
\end{abstract}

\section{Introduction}

In the totally asymmetric simple exclusion process (TASEP) on the one-dimensional integer lattice $\Z$,
particles with vacant right neighbour jump to the right by $1$ according to independent Poisson processes with unit rate.
However it is a simple non-reversible stochastic interacting particle system, the exclusion constraint produces an interesting behaviour.
There has been a lot of studies around this model and its discrete time versions.
Due to the determinantal structures of correlation functions,
the limiting process for particle positions or for the current fluctuations were found to be given by the Airy processes~\cite{Jo03b,BFPS06,BF07,Sas05}.

For a parameter $q\in[0,1)$, the $q$-TASEP is a particle system on $\Z$ where the jumps are independent of each other and happen with rate $1-q^{\rm gap}$
where the gap is the number of consecutive vacant sites next to the particle on its right.
It reduces to TASEP for $q=0$.
The $q$-TASEP was first introduced by Borodin and Corwin in~\cite{BC11}.
For step initial condition, a Fredholm determinant formula was given in~\cite{BCF14} for the $q$-Laplace transform of the particle position.
$q$-TASEP belongs to the Kardar--Parisi--Zhang (KPZ) universality class.
Based on the formula of~\cite{BCF14}, Ferrari and Vet\H o showed in~\cite{FV14} that the large time current fluctuations are governed by the (GUE) Tracy--Widom distribution.
This result confirms the KPZ universality conjecture for \mbox{$q$-TASEP} and shows that it possesses the characteristic asymptotic fluctuation statistics of the KPZ class.
The predictions of KPZ scaling theory conjecture are proved to be true in this case, see also \cite{Spo13a}.
A technical limitation of~\cite{FV14} was removed by Barraquand in~\cite{B14b} and the analysis was extended to the case of finitely many extra slow particles.
In~\cite{BC13}, two natural discrete time versions of \mbox{$q$-TASEP} were introduced
and Fredholm determinant expressions were proved for the $q$-Laplace transform of the particle positions.
The $q$-Boson particle system introduced by Sasamoto and Wadati in~\cite{SW98b} was proved to be dual to $q$-TASEP in~\cite{BCS12} and, as a consequence,
joint moment formulas for multiple particle positions were obtained for $q$-TASEP which characterize their distribution, however they are not of Fredholm determinant form.


Povolotsky introduced a three-parameter family of discrete time particle systems on $\Z$ in~\cite{P13} which was referred to as $q$-Hahn TASEP or $(q,\mu,\nu)$-TASEP in subsequent works.
This model is solvable by the Bethe ansatz, and many known integrable stochastic particle models can be obtained as limiting cases, in particular the $q$-TASEP, see Section~\ref{s:results}.
Using the duality of the $q$-Hahn Boson process and the $q$-Hahn TASEP,
Corwin derived a Fredholm determinant formula for the $q$-Laplace transform of the particle position in $q$-Hahn TASEP with step initial condition in~\cite{C14},
the proof was recently simplified by Barraquand in~\cite{B14a}.
This formula is used as a starting point of the asymptotic analysis carried out in the present paper, see Theorem~\ref{thm:finite} below.
The spectral theory for $q$-Hahn TASEP was developed in~\cite{BCPS14},
i.e.\ the eigenfunctions of the Markov transition operator and their properties are described.
The four-parameter family of stochastic higher spin vertex models in~\cite{CP15} also includes the $q$-Hahn TASEP.
In the $q$-Hahn asymmetric exclusion process which is a related two-sided continuous-time model, a discontinuity of the particle density and
Tracy--Widom asymptotics also for the first particles were found in~\cite{BC15}.

The asymptotic analysis performed in this paper shows similarities with the one in~\cite{BCF14} and \cite{FV14}.
In all of these cases, one of the main difficulties lies in the choice of the contours for the Fredholm determinant:
they are chosen to be steep descent paths which is sufficient for the asymptotic analysis to work, but also the extra singularities of the integrand have to be controlled.
The contours that we choose in the present case are circular and they are shown on Figure~\ref{fig:functioncontours}.
Since the present analysis covers a model with three parameters, the proof of the steep descent property along the contours here is more general and parallelly also more involved than in earlier works.
It was necessary however to impose the technical conditions \eqref{munucond}--\eqref{thetacond} on the parameters of the model
which excludes the application of the present results to the discrete time geometric $q$-TASEP, see~\cite{BC13}.

The paper is organized as follows.
We introduce the $q$-Hahn TASEP model and describe the main result on the fluctuation of the particle position in Section~\ref{s:results}.
Section~\ref{s:hydroKPZ} provides a physical explanation of global behaviour of the particle position, in particular, a heuristic proof of the law of large numbers is given.
Further, the prediction of the KPZ scaling conjecture on the non-universal scale coefficient is verified.
Section~\ref{s:finite} contains the pre-asymptotic Fredholm determinant formula for $q$-Hahn TASEP which was proved in~\cite{C14}.
We also show how the main result of the paper follows from the convergence of the corresponding Fredholm determinants.
The rest of the paper is devoted to the asymptotic analysis for the full proof of the limit theorem on the particle position variable:
Section~\ref{s:analysis} contains the main steps of the analysis as propositions;
the complex contours which are suitable for asymptotics are given and proved to be of steep descent in Section~\ref{s:steep};
finally the propositions are proved in Section~\ref{s:prop}.

\section{Model and main results}\label{s:results}

We start with the definition of the $q$-Hahn TASEP with \emph{step initial condition} following \cite{P13} and with further notations.
Let $q\in(0,1)$.
The \emph{$q$-Pochhammer symbol} is given by
\begin{equation}\label{defqpochhammer}
(a;q)_n=\prod_{k=0}^{n-1}(1-aq^k)
\end{equation}
for any $a\in\mathbb C$ and $n$ integer.
The definition naturally extends to the infinite $q$-Pochhammer symbol $(a;q)_\infty$ which is meant as an infinite product.
For a fixed $q\in(0,1)$ and $0<\nu<\mu<1$ and integers $0\le j\le m$, define the weights of the $q$-Binomial distribution as
\begin{equation}\label{defvarphi}
\varphi_{q,\mu,\nu}(j|m)=\mu^j\frac{(\nu/\mu;q)_j (\mu;q)_{m-j}}{(\nu;q)_m}\frac{(q;q)_m}{(q;q)_{j}(q;q)_{m-j}}.
\end{equation}
When $m=\infty$, extend this definition by setting
\begin{equation}\label{e.jumpdistinfty}
\varphi_{q,\mu,\nu}(j|\infty)=\mu^j\frac{(\nu/\mu;q)_j (\mu;q)_\infty}{(\nu;q)_\infty}\frac1{(q;q)_j}.
\end{equation}

The $q$-Hahn TASEP is a discrete time interacting particle system on $\Z$ with parallel updates
that consists of the evolution of particles ${\bf X}(\tau)=(X_N(\tau):N\in\Z\mbox{ or }N\in\mathbb N)$ for $\tau\ge0$.
The particles are numbered from right to left.
For the $N$th particle at time $\tau$, given that the number of vacant sites to the right of it is $m=X_{N-1}(\tau)-X_N(\tau)-1$,
the particle at $X_N(\tau)$ jumps to the right by $j$ with probability $\varphi_{q,\mu,\nu}(j|m)$ independently of the others.
Jumps of different particles happen with parallel updates.
For $\nu=0$, the dynamics reduces to the geometric $q$-TASEP, see~\cite{BC13},
and by setting $\mu=(1-q)\varepsilon$ and by scaling time by $\varepsilon^{-1}$, one gets the jump rates of $q$-TASEP as $\varepsilon\to0$.

Note that the dynamics preserves the order of particles.
Step initial condition means that the particles are initially at all negative integer positions, i.e.\ there are only particles with labels $N=1,2,\dots$ and they are initially at $X_N(0)=-N$.

\begin{definition}\label{def:qdigamma}
Let
\begin{equation}\label{defqgamma}
\Gamma_q(z)=(1-q)^{1-z}\frac{(q;q)_\infty}{(q^z;q)_\infty}
\end{equation}
be the \emph{$q$-gamma function}.
Then the \emph{$q$-digamma function} is defined by
\begin{equation}\label{defqdigamma}
\Psi_q(z)=\frac\partial{\partial z}\log\Gamma_q(z).
\end{equation}
\end{definition}

\begin{definition}\label{def:kfa}
Let $q\in(0,1)$ be fixed and choose a parameter $\theta>0$.
To these values, we associate the parameters
\begin{align}
\kappa\equiv \kappa(q,\mu,\nu,\theta)&=\frac{\Psi'_q(\theta)-\Psi'_q(\theta+\log_q\nu)}{\Psi'_q(\theta+\log_q\mu)-\Psi'_q(\theta+\log_q\nu)},\label{defkappa}\\
f\equiv f(q,\mu,\nu,\theta)&=\frac{\kappa\left(\Psi_q(\theta+\log_q\mu)-\Psi_q(\theta+\log_q\nu)\right)+\Psi_q(\theta+\log_q\nu)-\Psi_q(\theta)}{\log q},\label{deff}\\
\chi\equiv\chi(q,\mu,\nu,\theta)&=\frac{\kappa\left(\Psi''_q(\theta+\log_q\mu)-\Psi''_q(\theta+\log_q\nu)\right)+\Psi''_q(\theta+\log_q\nu)-\Psi''_q(\theta)}2,\label{defchi}\\
\phi\equiv\phi(q,\mu,\nu,\theta)&=\Psi_q(\theta+\log_q\mu)-\Psi_q(\theta+\log_q\nu).\label{defphi}
\end{align}
and we denote by $\phi'$ the derivative of $\phi$ with respect to $\theta$.
\end{definition}

\begin{remark}
Explicit formulas for the quantities given above are only available in terms of the parameter $\theta$ which appears naturally in the asymptotic analysis of the problem.
It could however be physically natural to parametrize the problem by $\kappa$ since it corresponds to the macroscopic position where we focus on as explained below.
Note that
\begin{equation}\label{kappachi}
\frac{\partial\kappa}{\partial\theta}=-\frac{2\chi}{\Psi'_q(\theta+\log_q\mu)-\Psi'_q(\theta+\log_q\nu)}
\end{equation}
holds by differentiation. The numerator is positive by Theorem~\ref{thm:kpz}
(the positivity of $\chi$ is actually proved at the end of Section~\ref{s:hydroKPZ} without using the rest of Theorem~\ref{thm:kpz}).
The denominator of \eqref{kappachi} is positive, since the function $z\mapsto\Psi_q'(z)$ is decreasing.
Hence $\kappa$ depends decreasingly on $\theta$ and the parametrization by $\kappa$ is also possible.
\end{remark}

The parameters $f$ and $\kappa$ describe the global behaviour of the particle system.
We provide a physical explanation of the following law of large numbers in Section~\ref{s:hydroKPZ}.
\begin{proposition}\label{prop:lln}
The law of large numbers
\begin{equation}\label{lln}
\frac{X_N(\tau=\kappa N)}{N} \to f-1
\end{equation}
holds for the position of the $N$th particle after time $\kappa N$ as $N\to\infty$.
\end{proposition}

In order to visualize the macroscopic behaviour given above, consider the evolution of the points $(X_N(\tau)+N,N)$ in the coordinate system.
For $\tau=0$, these points all lie on the positive half of the vertical axis.
For $\tau$ large and after rescaling the picture by $\tau$, the points are macroscopically around $(f/\kappa,1/\kappa)$
which is a curve that can be parametrized by $\theta$ and it is shown on Figure~\ref{fig:macro}.
By computing limits using \eqref{defkappa}--\eqref{deff}, one can see that the curve $(f/\kappa,1/\kappa)$ touches the axes
at $((\Psi_q(\log_q\mu)-\Psi_q(\log_q\nu))/\log q,0)$ for $\theta\to0$ and at $(0,(\mu-\nu)/(1-\nu))$ for $\theta\to\infty$.
It means that the right-most $q$-Hahn TASEP particle has speed $(\Psi_q(\log_q\mu)-\Psi_q(\log_q\nu))/\log q$
and that the left-most particle which has already started moving after time $\tau$ is around the position $-(\mu-\nu)\tau/(1-\nu)$.

\begin{figure}
\centering
\includegraphics[width=200pt]{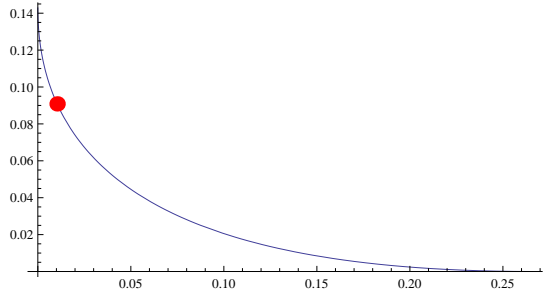}
\caption{The macroscopic shape of the positions of $q$-Hahn TASEP particles for $q=0.2,\mu=0.4,\nu=0.3$ which is given by the blue parametric curve $(f/\kappa,1/\kappa)$.
Due to the technical condition \eqref{thetacond}, Theorem~\ref{thm:main} is valid for positions which are on the right of the red dot on the plot,
i.e.\ for the major part of the rarefaction fan for these values of the parameters.\label{fig:macro}}
\end{figure}

In this paper, we study the fluctuations of particle $X_N$ around the deterministic macroscopic position given by \eqref{lln}.
One expects by KPZ universality that these fluctuations are of order $\O(N^{1/3})$ and have Tracy--Widom statistics (see the review~\cite{Fer10b}).
Further, at a given time $\tau=\kappa N$, particles are correlated over a scale $\O(N^{2/3})$ and their limit process is the Airy$_2$ process.
It is also expected by~\cite{CFP10b} that the same limit process arises for the position $X_N$ at times of order $N^{2/3}$ away from $\kappa N$ as it was shown for TASEP in~\cite{SI07}.

Therefore it is natural to consider for any $c\in\R$ the scaling
\begin{align}
\tau(N,c)&=\kappa N+cN^{2/3},\label{deftau}\\
p(N,c)&=(f-1)N-\frac{c\phi}{\log q}N^{2/3}+\frac{c^2(\phi')^2}{4\chi\log q}N^{1/3}\label{defp}
\end{align}
with $\kappa$, $f$, $\chi$, $\phi$ and $\phi'$ given in Definition~\ref{def:kfa}.
It means that on the top of the macroscopic behaviour given by \eqref{lln} and governed by the parameter $\theta$ through $\kappa$ and $f$,
we allow for a smaller $N^{2/3}$ time scale on which the parameter $c$ in \eqref{deftau} is understood as the time parameter of the expected Airy$_2$ process scaling limit.
Hence the rescaled tagged particle position given by
\begin{equation}\label{defxi}
\xi_N=\frac{X_N(\tau(N,c))-p(N,c)}{\chi^{1/3}(\log q)^{-1}N^{1/3}}
\end{equation}
is expected to converge to the Airy$_2$ process as a process in $c$.
Our main result is the convergence of the one-point distribution of $\xi_N$ to the Tracy--Widom distribution function~\cite{TW94}.

For our proof to work, we have to assume that for the parameters of the $q$-Hahn TASEP the technical conditions
\begin{gather}
q\le\nu<\mu\le1/2,\label{munucond}\\
\theta<\log_q\frac{2q}{1+q}\label{thetacond}
\end{gather}
hold.
It is shown on Figure~\ref{fig:macro} which part of the rarefaction fan is covered by the condition \eqref{thetacond} for a certain choice of parameters.
The upper bound \eqref{thetacond} on $\theta$ is plotted as a function of $q$ on Figure~\ref{fig:thetabound}.

\begin{figure}
\centering
\includegraphics[width=200pt]{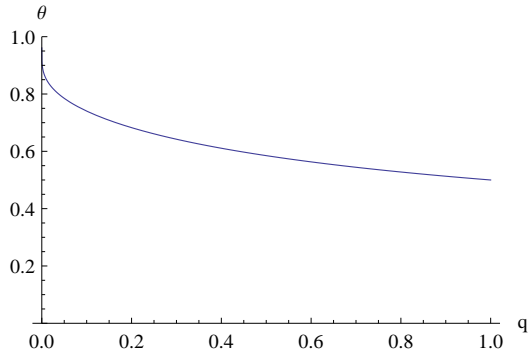}
\caption{The upper bound \eqref{thetacond} on $\theta$ as a function of $q$.\label{fig:thetabound}}
\end{figure}

\begin{theorem}\label{thm:main}
Let $q\in(0,1)$ and $\theta>0$ be fixed and suppose that the conditions \eqref{munucond}--\eqref{thetacond} hold.
For any $c,x\in\R$ and with the notation above, the rescaled position $\xi_N$ converges in distribution, i.e.
\begin{equation}
\lim_{N\to\infty}\P(\xi_N<x)=F_{\rm GUE}(x)
\end{equation}
where $F_{\rm GUE}$ is the GUE Tracy--Widom distribution function.
\end{theorem}

\begin{remark}
Condition \eqref{munucond} is needed for the proof of Proposition~\ref{prop:steep1} and \ref{prop:steep2}
to establish the steep descent property along the contours $\C_\theta$ and $\D_\theta$.
The origin of this condition is more explained in Remark~\ref{rem:munucond}.
We expect that this condition is purely technical,
because the predictions of the KPZ scaling conjecture (explained in Section~\ref{s:hydroKPZ} in details) are valid in the full parameter range.

Condition \eqref{thetacond} is already used in the first step of the proof of Proposition~\ref{prop:localization} in the contour deformation
in order to make sure that no poles coming from the sine in the denominator of the integrand in the kernel $K_x$ given by \eqref{defKx} have to be encountered.
Theorem~\ref{thm:main} is expected to hold in the entire rarefaction fan without the technical limitation \eqref{thetacond}.
To eliminate this condition, one should control the blow up of the kernel at the additional poles by futher bounds on the kernel.
The restriction of \eqref{thetacond} is shown on Figure~\ref{fig:thetabound}.
Note that \eqref{thetacond} is in particular satisfied for $\theta<1/2$ and for any $q\in(0,1)$.
\end{remark}

\begin{remark}
An equivalent statement of Theorem~\ref{thm:main} is that if one expresses the particle position as
\begin{equation}\label{limthmwithtau}
X_{N(\tau,c)}(\tau)=P(\tau,c)+\frac{\chi^{1/3}}{\kappa^{1/3}\log q}\xi_\tau \tau^{1/3},
\end{equation}
where
\begin{align*}
N(\tau,c)&=\frac\tau\kappa-\frac{c\tau^{2/3}}{\kappa^{5/3}}+\frac{2c^2\tau^{1/3}}{3\kappa^{7/3}},\\
P(\tau,c)&=\frac{f-1}\kappa \tau-c\left(\frac{f-1}{\kappa^{5/3}}+\frac\phi{\kappa^{2/3}\log q}\right)\tau^{2/3}\\
&\qquad+c^2\left(\frac{2(f-1)}{3\kappa^{7/3}}+\frac{2\phi}{3\kappa^{4/3}\log q}+\frac{(\phi')^2}{4\chi\kappa^{1/3}\log q}\right)\tau^{1/3},
\end{align*}
then for any $x\in\R$,
\[\lim_{\tau\to\infty}\P(\xi_\tau<x)=F_{\rm GUE}(x).\]
\end{remark}

To state the next equivalent formulation of our main result, we introduce the height function $h(j,\tau)$ via the height differences
$h(j+1,\tau)=h(j,\tau)=-1$ if there is a particle at position $j$ at time $\tau$ and $+1$ if the site is vacant.
This defines the height function up to a global shift which is determined by its initial value $h(j,0)=|j|$.
It corresponds to the step (or wedge) initial condition.
We remark that the height function can also be interpreted as particle current:
the number of particles in the interval $[j,\infty)$ at time $\tau$ is given by $(h(j,\tau)-j)/2$.

\begin{theorem}\label{thm:kpz}
Let $q\in(0,1)$ and $\theta>0$ be fixed and suppose that the conditions \eqref{munucond}--\eqref{thetacond} hold.
With the notation introduced above, we can write the height fluctuation as
\begin{equation}\label{hfluct}
h\left(\frac{f-1}\kappa \tau,\tau\right)=\frac{f+1}\kappa \tau+\frac2{\log q+\Psi_q(\theta)-\Psi_q(\theta+\log_q\nu)}\frac{\chi^{1/3}}{\kappa^{1/3}}\xi_\tau\tau^{1/3}
\end{equation}
with
\[\lim_{\tau\to\infty}\P(\xi_\tau<x)=F_{\rm GUE}(x)\]
for any $x\in\R$.

Furthermore, this verifies the prediction of the KPZ scaling theory conjecture on the non-universal scale coefficient of the $\xi_\tau$.
In particular, this coefficient is negative due to the fact that $\chi>0$ for any choice of $q\in(0,1)$ and $0<\nu<\mu<1$.
\end{theorem}

\section{Hydrodynamic limit and KPZ scaling conjecture}\label{s:hydroKPZ}

In this section, we first give a non-rigorous argument about the hydrodynamic limit of $q$-Hahn TASEP, more precisely, about the particle density
\[\rho(t,x)=\lim_{\tau\to\infty}\P(\mbox{there is a particle at position $x\tau$ at time $t\tau$})\]
where the limit is expected to exist.
In particular, we show that if the local stationarity assumption is satisfied, then the law of large numbers stated in Proposition~\ref{prop:lln} holds.
This computation is based on the mass conservation of particles.
Later, in Section~\ref{s:analysis}, we prove Theorem~\ref{thm:main} using the steepest descent method,
which implies the law of large numbers without the local stationarity assumption.
Further, the KPZ scaling theory provides a prediction on the non-universal scale coefficient.
We show in this section that the predictions coincide with the coefficients which appear in Theorem~\ref{thm:main} and~\ref{thm:kpz}.

It was shown in~\cite{EMZ04} that the $q$-Hahn TASEP admits a family of translation invariant stationary measures parametrized by $\alpha\in(0,1)$
where the gaps between consecutive particles are i.i.d.\ random variables with distribution
\begin{equation}\label{defG}
\P(G=k)=\frac{(\alpha;q)_\infty}{(\alpha\nu;q)_\infty}\frac{(\nu;q)_k}{(q;q)_k}\alpha^k
\end{equation}
for $k=0,1,2,\dots$ which is a proper probability distribution by the $q$-Binomial theorem
\begin{equation}\label{qbinom}
\sum_{n=0}^\infty\frac{(a;q)_n}{(q;q)_n}z^n=\frac{(az;q)_\infty}{(z;q)_\infty}.
\end{equation}
For the stationary $q$-Hahn TASEP, the particle density is clearly a constant $\rho(t,x)\equiv\rho$, its value is given in Proposition~\ref{prop:stationary} below.
The average particle current $j(\rho)$ in the stationary $q$-Hahn TASEP is defined as the probability that a bond of $\Z$ is jumped over by a particle in one time step.
This definition does not depend on the choice of the bond, it is a function of the parameter $\alpha$ of the stationary measure \eqref{defG},
hence by Remark~\ref{rem:rhoinverse}, it is a function of the density $\rho$.

\begin{proposition}\label{prop:stationary}
The stationary measure of $q$-Hahn TASEP with parameter $\alpha$ has particle density and particle current given by
\begin{equation}\label{defrhojrho}
\rho=\frac{\log q}{\log q+\Psi_q(\log_q\alpha)-\Psi_q(\log_q(\alpha\nu))},\quad
j(\rho)=\frac{\Psi_q(\log_q(\alpha\mu))-\Psi_q(\log_q(\alpha\nu))}{\log q+\Psi_q(\log_q\alpha)-\Psi_q(\log_q(\alpha\nu))}.
\end{equation}
\end{proposition}

\begin{remark}\label{rem:rhoinverse}
Since the function $z\mapsto\Psi_q(z)$ is increasing for $z>0$, the numerators and denominators in \eqref{defrhojrho} are all negative.
Note that the function $z\mapsto\Psi_q(z)-\Psi_q(z+\log_q\nu)$ is also increasing for $z>0$ since $\log_q\nu>0$.
As a consequence, for fixed $q$, $\mu$ and $\nu$, the formula for $\rho$ in \eqref{defrhojrho} is a decreasing function of $\alpha$.
Using the inverse of this function, the average particle current can be understood as a function of the particle density $\rho$.
\end{remark}

Since the number of particles in $q$-Hahn TASEP is a conserved quantity, the following mass conservation equation has to be satisfied provided that the particle density is well-defined:
\begin{equation}\label{pde}
\frac\partial{\partial t}\rho(t,x)+\frac\partial{\partial x}j(\rho(t,x))=0.
\end{equation}
If the $q$-Hahn TASEP is started from step initial condition, then the corresponding initial condition for \eqref{pde} is $\rho(0,x)=\id(x<0)$.

\begin{proof}[Proof of Proposition~\ref{prop:stationary}]
Since in the stationary $q$-Hahn TASEP the gaps between particles are i.i.d.\ and distributed as $G$ given in \eqref{defG} with some fixed $\alpha\in(0,1)$, the renewal theorem implies that
\begin{equation}\label{rhoG}
\rho=\frac1{1+\E(G)}.
\end{equation}
A direct computation yields
\begin{align*}
\E(G)&=\frac{(\alpha;q)_\infty}{(\alpha\nu;q)_\infty}\sum_{k=0}^\infty k\frac{(\nu;q)_k}{(q;q)_k}\alpha^k\\
&=\frac{(\alpha;q)_\infty}{(\alpha\nu;q)_\infty}\alpha\frac\partial{\partial\alpha}\frac{(\alpha\nu;q)_\infty}{(\alpha;q)_\infty}\\
&=\alpha\frac\partial{\partial\alpha}\log\frac{(\alpha\nu;q)_\infty}{(\alpha;q)_\infty}\\
&=\frac{\Psi_q(\log_q\alpha)-\Psi_q(\log_q(\alpha\nu))}{\log q}
\end{align*}
which together with \eqref{rhoG} proves the formula for $\rho$.
The average particle current is the product of the particle density $\rho$ and the expected jump size.
The latter is equal to
\begin{align*}
\sum_{n=0}^\infty\P(G=n)\sum_{m=0}^n m\varphi(m|n)&=\sum_{m=0}^\infty\sum_{k=0}^\infty\P(G=m+k)m\varphi(m|m+k)\\
&=\frac{(\alpha;q)_\infty}{(\alpha\nu;q)_\infty}\sum_{m=0}^\infty m\mu^m\frac{(\nu/\mu;q)_m}{(q;q)_m}\sum_{k=0}^\infty \alpha^{m+k}\frac{(\mu;q)_k}{(q;q)_k}\\
&=\frac{(\alpha\mu;q)_\infty}{(\alpha\nu;q)_\infty}\alpha\frac\partial{\partial\alpha}\sum_{m=0}^\infty(\alpha\mu)^m\frac{(\nu/\mu;q)_m}{(q;q)_m}\\
&=\alpha\frac\partial{\partial\alpha}\sum_{m=0}^\infty\log\frac{(\alpha\nu;q)_\infty}{(\alpha\mu;q)_\infty}\\
&=\frac{\Psi_q(\log_q(\alpha\mu))-\Psi_q(\log_q(\alpha\nu))}{\log q}
\end{align*}
where we used the definitions \eqref{defG} and \eqref{defvarphi} in the second equation and the $q$-Binomial theorem \eqref{qbinom} in the third equation.
The previous calculation verifies the expression for $j(\rho)$ in \eqref{defrhojrho}.
\end{proof}

\begin{proof}[Heuristic proof of Proposition~\ref{prop:lln}]
The proof that we give here assumes the local stationarity of the particle system,
that is, the gaps between particles are distributed like \eqref{defG} for some space and time dependent parameter $\alpha$ which can be obtained from the local particle density.
The full proof of Proposition~\ref{prop:lln} follows from Theorem~\ref{thm:main}.

To get the hydrodynamic limit, position $p(\theta)\tau$ is considered after time $\tau$ for large $\tau$ for some parametric global position $p(\theta)$.
By local stationarity, the gaps between particles are distributed around this position given by \eqref{defG} for some $\alpha$ and we assume that the parametrization is such that $\alpha=q^\theta$ holds.
The relation of the macroscopic position $p(\theta)$ and $\alpha$ is therefore determined by the mass conservation PDE \eqref{pde} as follows.
By the stationarity assumption, we can use formulas \eqref{defrhojrho} to get
\begin{equation}\label{hydrosolution}\begin{aligned}
\rho(t,p(\theta)t)&=\frac{\log q}{\log q+\Psi_q(\theta)-\Psi_q(\theta+\log_q\nu)},\\
j(\rho(t,p(\theta)t))&=\frac{\Psi_q(\theta+\log_q\mu)-\Psi_q(\theta+\log_q\nu)}{\log q+\Psi_q(\theta)-\Psi_q(\theta+\log_q\nu)}
\end{aligned}\end{equation}
where $\rho$ with double argument is the solution of the PDE \eqref{pde}.
The function $p(\theta)$ can be expressed from the derivatives of the quantities in \eqref{hydrosolution} and using the PDE \eqref{pde}.
The solution is given in~\cite{BC15} in general simply by
\begin{equation}\label{pfrac}
p(\theta)=\frac{\partial j(\rho(t,p(\theta)))}{\partial\theta}\left(\frac{\partial \rho(t,p(\theta)t)}{\partial\theta}\right)^{-1}
\end{equation}
which in our case equals $(f-1)/\kappa$ with \eqref{defkappa}--\eqref{deff} by straightforward computation from \eqref{hydrosolution}.

Let $r(\theta)\tau$ be the label of particle at position $p(\theta)\tau$ at time $\tau$ in leading order for some function $r(\theta)$.
By definition, one has
\begin{equation}\label{rpderivative}
\frac{\d r(\theta)}{\d p(\theta)}=-\rho(t,p(\theta)t)
\end{equation}
with $r(\infty)=-p(\infty)$ by the step initial condition.
The equation \eqref{rpderivative} is satisfied by
\begin{align*}
r(\theta)&=-p(\infty)+\int_\theta^\infty \rho(t,p(\theta')t)\frac{\d p(\theta')}{\d\theta'}\d\theta'\\
&=-p(\infty)+[\rho(t,p(\theta')t)p(\theta')-j(\rho(t,p(\theta')t))]_\theta^\infty\\
&=-\rho(t,p(\theta)t)p(\theta)+j(\rho(t,p(\theta)t))
\end{align*}
where integration by parts and \eqref{pfrac} are used in the second equation as it is given in~\cite{BC15} in general.
Using the fact that $p(\theta)=(f-1)/\kappa$ and \eqref{hydrosolution}, it is straightforward to see that $r(\theta)=1/\kappa$.
This verifies the law of large numbers \eqref{lln}.
\end{proof}

At the end of this section, we show how Theorem~\ref{thm:kpz} follows from Theorem~\ref{thm:main} and we verify the validity of the KPZ scaling theory conjecture.

\begin{proof}[Proof of Theorem~\ref{thm:kpz}]
We start with \eqref{limthmwithtau} for $c=0$ which reads as
\begin{equation}\label{limthmsimple}
X_{\tau/\kappa}(\tau)=\frac{f-1}\kappa \tau+\frac{\chi^{1/3}}{\kappa^{1/3}\log q}\xi_\tau\tau^{1/3}
\end{equation}
where $\xi_\tau$ is asymptotically Tracy--Widom distributed.
The global position $(f-1)/\kappa$ and the global factor $1/\kappa$ of the particle label are both functions of $\theta$.
We consider the right-hand side of \eqref{limthmsimple} as the $\tau^{-2/3}$ order random perturbation of the function $(f-1)/\kappa$.
With a $\tau^{-2/3}$ order modification of $\theta$, the right-hand side of \eqref{limthmsimple} becomes $(f-1)/\kappa$
and the random perturbation appears in the particle label with coefficient multiplied by
\[-\frac\d{\d\theta}\frac1\kappa\left(\frac\d{\d\theta}\frac{f-1}\kappa\right)^{-1}=-\frac{\d r(\theta)}{\d p(\theta)}=\rho(t,p(\theta)t)=\frac{\log q}{\log q+\Psi_q(\theta)-\Psi_q(\theta+\log_q\nu)}\]
with the notations of the previous proof and by using \eqref{rpderivative} and \eqref{hydrosolution}.
This means that \eqref{limthmsimple} is transformed into
\begin{equation}\label{limthmNtilde}
X_{\wt N}(\tau)=\frac{f-1}\kappa \tau\quad\mbox{with}\quad\wt N=\frac\tau\kappa+\frac1{\log q+\Psi_q(\theta)-\Psi_q(\theta+\log_q\nu)}\frac{\chi^{1/3}}{\kappa^{1/3}}\xi_\tau\tau^{1/3}.
\end{equation}

By the remark before Theorem~\ref{thm:kpz}, we have the equality of events
\[\{X_N(\tau)<j\}=\left\{\frac{h(j,\tau)-j}2<N\right\}\]
which exactly implies \eqref{hfluct} using \eqref{limthmNtilde}.

Now we prove that the KPZ scaling theory conjecture is satisfied,
more precisely, we show that the non-universal scale coefficient in \eqref{hfluct} coincides with the prediction given in~\cite{Spo13a}.
We denote the particle density and the corresponding particle current given in \eqref{defrhojrho} by $\rho$ and $j$ respectively for the sake of simplicity.
The next two crucial quantities are defined in order to verify the KPZ conjecture of~\cite{Spo13a}.
First let
\begin{equation}\label{deflambda}
\lambda=\frac12\frac{\d^2j}{\d\rho^2}.
\end{equation}
We remark that the $\rho$ in~\cite{Spo13a} corresponds to $2\rho-1$ with our notation by the relation of the particle density and the slope of the height function $h$.
Consequently, the current in~\cite{Spo13a} equals to $-2j$ with the present notation the minus sign being present due to the fact that all particles jump to the left in~\cite{Spo13a}.
This explains the extra $-1/2$ factor in \eqref{deflambda} compared to~\cite{Spo13a}.

To define the second quantity, it is supposed that the stationary distribution of the gaps between particles is of the form
\[\P(G=k)=\frac1{Z(\alpha)}\left(\prod_{j=1}^k g(j)\right)^{-1}\alpha^k\]
for some increasing function $g$ where
\[Z(\alpha)=\sum_{k=0}^\infty\left(\prod_{j=1}^k g(j)\right)^{-1}\alpha^k\]
is the normalizing constant and $G(\alpha)=\log Z(\alpha)$.
Then the second quantity of interest is
\begin{equation}\label{defA}
A=\frac{4\alpha(\alpha G'(\alpha))'}{(1+\alpha G'(\alpha))^3}
\end{equation}
where the primes denote derivatives with respect to $\alpha$.
The scaling conjecture predicts the non-universal scale coefficient to be $-(-\frac12\lambda A^2\tau)^{1/3}$.

We check the conjecture by direct computation.
To simplify the calculation in this proof, we introduce the shorthand notation
\[a=\log q+\Psi_q(\theta)-\Psi_q(\theta+\log_q\nu),\qquad b=\Psi_q(\theta+\log_q\mu)-\Psi_q(\theta+\log_q\nu)\]
which we understand as functions of $\theta$.
With this notation,
\begin{equation}\label{rhojab}
\rho=\log q/a,\qquad j=b/a
\end{equation}
by \eqref{hydrosolution}.
In our case, the quantity $\lambda$ is computed as
\[\lambda=\frac{j''\rho'-j'\rho''}{2(\rho')^3}=\frac{a^3(b''a'-a''b')}{2(a')^3(\log q)^2}\]
where all the primes are derivatives with respect to $\theta$ and we used \eqref{deflambda} and \eqref{rhojab}.
For $q$-Hahn TASEP, $G(\alpha)=\log(\alpha\nu;q)_\infty-\log(\alpha;q)_\infty$ and the relation $\alpha=q^\theta$ holds as observed in the proof of Proposition~\ref{prop:lln},
hence differentiation as given by \eqref{defA} yields that
\[A=\frac{4a'\log q}{a^3}\]
where the prime is derivative with respect to $\theta$.
This gives that
\[-\frac12\lambda A^2=-\frac{4(b''a'-a''b')}{a'a^3}=-\frac{8\chi}{a^3\kappa}\]
where we used the observations $b''a'-a''b'=2\chi b'$ and $\kappa=a'/b'$ in the second equality.
This completes the proof of the scaling conjecture for $q$-Hahn TASEP.

Finally, we show the positivity of the parameter $\chi$.
By definition \eqref{defkappa}--\eqref{defchi} and due to the fact that the function $\Psi_q'$ is decreasing, the positivity of $\chi$ is equivalent to
\[\frac{\Psi_q''(\theta+\log_q\mu)-\Psi_q''(\theta+\log_q\nu)}{\Psi_q'(\theta+\log_q\mu)-\Psi_q'(\theta+\log_q\nu)}
>\frac{\Psi_q''(\theta)-\Psi_q''(\theta+\log_q\nu)}{\Psi_q'(\theta)-\Psi_q'(\theta+\log_q\nu)}.\]
We will prove that the function
\[x\mapsto\frac{\Psi_q''(\theta+x)-\Psi_q''(\theta+\log_q\nu)}{\Psi_q'(\theta+x)-\Psi_q'(\theta+\log_q\nu)}\]
is increasing for $x\in(0,\log_q\nu)$.
By taking derivative and using the fact that $\Psi_q'$ is decreasing and $\Psi_q''<0$, what remains to show is
\[\frac{\Psi_q'''(\theta+x)}{\Psi_q''(\theta+x)}<\frac{\Psi_q''(\theta+x)-\Psi_q''(\theta+\log_q\nu)}{\Psi_q'(\theta+x)-\Psi_q'(\theta+\log_q\nu)}.\]
By Cauchy's mean value theorem, the right-hand side above is equal to $\Psi_q'''(\theta+y)/\Psi_q''(\theta+y)$ for some $y\in(x,\log_q\nu)$.
Hence the proof is complete, since $\theta\mapsto\Psi_q'''(\theta)/\Psi_q''(\theta)$ is an increasing function, see e.g.\ the proof of Lemma 4.2 in~\cite{BC15}.
\end{proof}

\section{Finite time formula and proof of the main result}\label{s:finite}

The first part of Theorem 1.10 in~\cite{C14} gives the following Fredholm determinant expression for the $q$-Laplace transform of the particle position in $q$-Hahn TASEP with step initial condition.

\begin{theorem}\label{thm:finite}
Fix $q\in (0,1)$ and $0\leq \nu\leq \mu<1$.
Consider $q$-Hahn TASEP $(X_N(\tau))_{N\ge1}$ started from step initial data.
Then for all $\zeta\in\mathbb C\setminus \R_+$,
\begin{equation}\label{MellinBarnes}
\E\left(\frac1{(\zeta q^{X_{N}(\tau)+N};q)_{\infty}}\right)=\det\left(I + K_{\zeta}\right)_{L^2(C_1)}
\end{equation}
where $C_1$ is a positively oriented circle containing 1 with small enough radius so as to not contain 0, $1/q$ and $1/\nu$.
The operator $K_\zeta$ is defined in terms of its integral kernel
\begin{equation}\label{defKzeta}
K_{\zeta}(w,w') = \frac{1}{2\pi\I} \int_{\frac12+\I\R} \frac{\pi}{\sin(-\pi s)} (-\zeta)^s \frac{h(w)}{h(q^s w)} \frac{1}{q^s w - w'} \d s
\end{equation}
with
\[h(w) = \left(\frac{(\nu w;q)_{\infty}}{(w;q)_{\infty}}\right)^N \left( \frac{(\mu w;q)_{\infty}}{(\nu w;q)_{\infty}}\right)^\tau \frac{1}{(\nu w;q)_{\infty}}.\]
\end{theorem}

Let us choose
\begin{equation}\label{defzeta}
\zeta=-q^{-fN+\frac{c\phi}{\log q}N^{2/3}+\beta_x\frac{N^{1/3}}{\log q}}\in\mathbb C\setminus\R_+
\end{equation}
where
\begin{equation}\label{defbeta}
\beta_x=\frac{c^2(\phi')^2}{4\chi}-\chi^{1/3}x.
\end{equation}
Theorem~\ref{thm:Fredholmconv} below is about the convergence of the Fredholm determinant on the right-hand side of \eqref{MellinBarnes} under the right scaling of the parameters.
It is the most important input for the Tracy--Widom limit of the rescaled particle position in $q$-Hahn TASEP.
It is proved by the method of steepest descent later in Section~\ref{s:analysis}.
We show in this section how the proof of Theorem~\ref{thm:main}, the main result of this paper follows from Theorem~\ref{thm:Fredholmconv}.

\begin{theorem}\label{thm:Fredholmconv}
Let $x\in\R$ be fixed and choose $\zeta$ according to \eqref{defzeta}.
Let $\tau$ be scaled with $N$ as in \eqref{deftau}.
Suppose that for the parameters of the $q$-Hahn TASEP, the conditions \eqref{munucond}--\eqref{thetacond} hold.
Then
\[\det(\id+K_\zeta)_{L^2(C_1)}\to F_{\rm GUE}(x)\]
as $N\to\infty$.
\end{theorem}

\begin{proof}[Proof of Theorem~\ref{thm:main}]
With the scaling \eqref{defzeta} of $\zeta$ on the left-hand side of \eqref{MellinBarnes}, one has
\begin{equation}\label{zetaeq}
\zeta q^{X_{N}(\tau)+N}=-q^{\frac{\chi^{1/3}}{\log q}N^{1/3}(\xi_N-x)}.
\end{equation}
Hence that the argument presented in Section 5 of~\cite{FV14} can be used for the $q$-Hahn TASEP as well.
In particular, using Lemma 5.1 of~\cite{FV14}, it follows from \eqref{zetaeq} that the left-hand side of \eqref{MellinBarnes}
converges as $N\to\infty$ to the limiting distribution function $\lim_{N\to\infty}\P(\xi_N<x)$ when $\zeta$ is rescaled via \eqref{defzeta}.
By Lemma 4.1.39 of~\cite{BC11}, Theorem~\ref{thm:Fredholmconv} on the convergence of the right-hand side of \eqref{MellinBarnes}
to the GUE Tracy--Widom distribution function is enough for the weak convergence of $\xi_N$ and for the proof of Theorem~\ref{thm:main}.
\end{proof}

\section{Asymptotic analysis}\label{s:analysis}

We prove Theorem~\ref{thm:Fredholmconv} in this section.
In order to perform the asymptotic analysis, we substitute \eqref{deftau} and \eqref{defzeta} for the values of $\tau$ and $\zeta$ into \eqref{defKzeta} and perform the change of variables
\begin{equation}\label{wWsZ}
w=q^W,\qquad w'=q^{W'},\qquad s+W=Z.
\end{equation}
The kernel which we get is
\begin{multline}\label{defKx}
K_x(W,W')\\
=\frac{q^W\log q}{2\pi\I}\int_{\theta+\I\R}\frac{\d Z}{q^Z-q^{W'}}\frac\pi{\sin(\pi(W-Z))}\frac{(\nu q^Z;q)_\infty}{(\nu q^W;q)_\infty}
\frac{e^{Nf_0(q^Z)+N^{2/3}f_1(q^Z)+N^{1/3}f_2(q^Z)}}{e^{Nf_0(q^W)+N^{2/3}f_1(q^W)+N^{1/3}f_2(q^W)}}
\end{multline}
with
\begin{align}
f_0(z)&=-f\log z+\kappa(\log(\nu z;q)_\infty-\log(\mu z;q)_\infty)+\log(z;q)_\infty-\log(\nu z;q)_\infty,\label{deff0}\\
f_1(z)&=c\phi\log_q z+c(\log(\nu z;q)_\infty-\log(\mu z;q)_\infty),\label{deff1}\\
f_2(z)&=\beta_x\log_q z\label{deff2}
\end{align}
and $\beta_x$ as in \eqref{defbeta}.
The contours for the Fredholm determinant and for the integral defining the kernel transform under the change of variables \eqref{wWsZ} are as follows.
The contour for $w$ and $w'$ was originally $C_1$, a small circle around $1$, hence the contour for $W$ and $W'$ can be chosen to be $C_0$
which is a small circle around $0$ that does not contain the singularities at $-1$ and at $-\log_q\nu$.
If this circle is small enough, the contour for $Z$ becomes a small perturbation of $1/2+\I\R$ which can be shifted to $\theta+\I\R$ without crossing any singularity of the integrand
since \eqref{thetacond} means in particular that $\theta\in(0,1)$.
Hence the choice for the $Z$ contour in \eqref{defKx} is appropriate and we can write the equality of the Fredholm determinants
\[\det(\id+K_\zeta)_{L^2(C_1)} = \det(\id+K_x)_{L^2(C_0)}.\]

Theorem~\ref{thm:Fredholmconv} follows immediately from the series of propositions below.
The propositions are stated in this section without proofs in order to keep the proof of Theorem~\ref{thm:Fredholmconv} transparent.
The proofs of the propositions are given later separately in Section~\ref{s:prop}.
In the propositions without repeating everywhere, we assume that for the parameters of the $q$-Hahn TASEP, the conditions \eqref{munucond}--\eqref{thetacond} hold.
To state the first proposition, we introduce the V-shaped contour
\begin{equation}\label{defV}
V_{\theta,\varphi}^\delta=\{\theta+e^{\I\varphi\sgn(t)}|t|:t\in[-\delta,\delta]\}
\end{equation}
where $\theta>0$ is the tip of the V, $\varphi\in(0,\pi)$ is its angle and $\delta\in\R_+\cup\{\infty\}$.
We also introduce the kernel
\begin{multline}\label{defKxdelta}
K_{x,\delta}(W,W')\\
=\frac{q^W\log q}{2\pi\I}\int_{V_{\theta,\varphi}^\delta}\frac{\d Z}{q^Z-q^{W'}}\frac\pi{\sin(\pi(W-Z))}\frac{(\nu q^Z;q)_\infty}{(\nu q^W;q)_\infty}
\frac{e^{Nf_0(q^Z)+N^{2/3}f_1(q^Z)+N^{1/3}f_2(q^Z)}}{e^{Nf_0(q^W)+N^{2/3}f_1(q^W)+N^{1/3}f_2(q^W)}}
\end{multline}
where $W,W'\in V_{\theta,\pi-\varphi}^\delta$.
The dependence of the kernel on $\varphi$ is not indicated in the notation.
Note that $K_{x,\delta}$ only differs from $K_x$ by the integration contours.

\begin{proposition}\label{prop:localization}
For any fixed $\delta>0$ and $\varepsilon>0$ small enough, there are $\varphi\in(0,\pi/2)$ and $N_0$ such that for all $N>N_0$
\[\left|\det(\id+K_x)_{L^2(C_0)}-\det(\id-K_{x,\delta})_{L^2(V_{\theta,\pi-\varphi}^\delta)}\right|<\varepsilon.\]
\end{proposition}

By defining the rescaled kernel
\begin{equation}\label{defKxrescaled}
K_{x,\delta}^N(w,w')=N^{-1/3}K_{x,\delta N^{1/3}}(\theta+wN^{-1/3},\theta+w'N^{-1/3}),
\end{equation}
the change of variables shows that
\[\det(\id-K_{x,\delta})_{L^2(V_{\theta,\pi-\varphi}^\delta)}=\det(\id-K_{x,\delta}^N)_{L^2(V_{0,\pi-\varphi}^{\delta N^{1/3}})}.\]
Next we show that on the contour $V_{0,\pi-\varphi}^{\delta N^{1/3}}$,
the kernel $K_{x,\delta}^N$ can be replaced by the one obtained by using the Taylor approximation given later in \eqref{f0Taylor}--\eqref{f2Taylor}.

\begin{proposition}\label{prop:kernelconv}
For any fixed $\varepsilon>0$ small enough, there is a small $\delta>0$ and an $N_0$ such that for any $N>N_0$,
\[\left|\det(\id-K_{x,\delta}^N)_{L^2(V_{0,\pi-\varphi}^{\delta N^{1/3}})}-\det(\id-K_{x,\delta N^{1/3}}')_{L^2(V_{0,\pi-\varphi}^{\delta N^{1/3}})}\right|<\varepsilon\]
where
\begin{equation}\label{kernelK'}
K_{x,L}'(w,w')=\frac1{2\pi\I}\int_{V_{0,\varphi}^L}\frac{\d z}{(z-w')(w-z)}\frac{e^{\chi z^3/3+c\phi'z^2/2+\beta_x z}}{e^{\chi w^3/3+c\phi'w^2/2+\beta_x w}}.
\end{equation}
\end{proposition}

\begin{proposition}\label{prop:kernelextend}
With the notation as above,
\[\det(\id-K_{x,\delta N^{1/3}}')_{L^2(V_{0,\pi-\varphi}^{\delta N^{1/3}})}\to\det(\id-K_{x,\infty}')_{L^2(V_{0,\pi-\varphi}^\infty)}\]
as $N\to\infty$.
\end{proposition}

\begin{proposition}\label{prop:rewritekernel}
We can rewrite the Fredholm determinant
\[\det(\id-K_{x,\infty}')_{L^2(V_{0,\pi-\varphi}^{\delta N^{1/3}})}=\det(\id-K_{\Ai,x})_{L^2(\R_+)}=F_{\rm GUE}(x)\]
where
\[K_{\Ai,x}(a,b)=\int_0^\infty\d\lambda\Ai(x+a)\Ai(x+b)\]
and $F_{\rm GUE}$ is the GUE Tracy--Widom distribution function.
\end{proposition}

\section{Steep descent contours}\label{s:steep}

This section is devoted to establish contours which are of steep descent for the function with principal contribution in the exponent.
We follow the lines of the method of steepest descent with the slight generalization that our contours are not necessarily of steepest descent but of steep descent.
A finite contour $\gamma$ in the complex plane is of steep descent for a real function
if there is a unique point on $\gamma$ where the function attains its maximum over $\gamma$ and there is another point on $\gamma$ where the function is minimal,
furthermore, the function is monotone along both arcs connecting the two points.
We first define these contours which are also shown on Figure~\ref{fig:functioncontours} along with the contourplot of the function $\Re(f_0)$.

\begin{definition}\label{def:contours}
Let us define the contours along with their parametrizations as
\[\C_\theta=\{w(s)=1-(1-q^\theta)e^{\I s},s\in(-\pi,\pi]\},\qquad
\D_\theta=\{z(t)=q^\theta e^{\I t},t\in(-\pi,\pi]\}\]
and let $\wt\C_\theta$ be the image of $\C_\theta$ under the map $w\mapsto\log_q w$.
\end{definition}

\begin{figure}
\centering
\includegraphics[height=200pt]{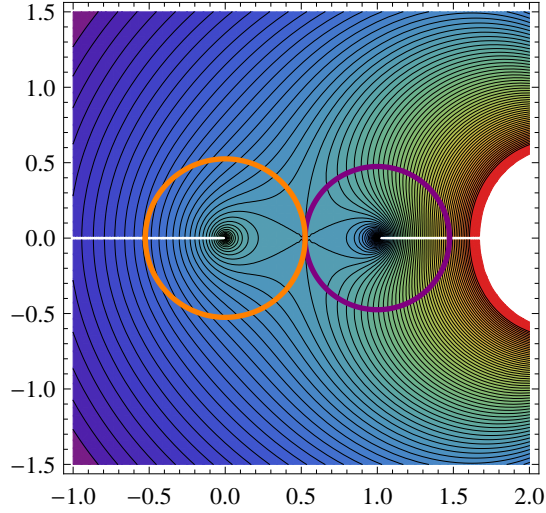}
\caption{Contourplot of the function $\Re(f_0)$ in the neighbourhood of the double critical point $\theta$ and the steep descent contours $\C_\theta$ (purple) and $\D_\theta$ (orange).
Parameter values: $\theta=0.4$, $q=0.2$, $\mu=0.4$, $\nu=0.3$.\label{fig:functioncontours}}
\end{figure}

The next two propositions are the main statements of the present section and the main technical tools for the proof of Theorem~\ref{thm:main}.
They are proved later in this section after the proof of Lemma~\ref{lemma:compareg}.

\begin{proposition}\label{prop:steep1}
Suppose that for the parameters of the $q$-Hahn TASEP, \eqref{munucond} holds.
Then the contour $\C_\theta$ is of steep descent for the function $-\Re(f_0)$
in the sense that the function attains its maximum at $q^\theta$ corresponding to $s=0$, it increases for $s\in(-\pi,0)$ and it decreases for $s\in(0,\pi)$.
\end{proposition}

\begin{proposition}\label{prop:steep2}
Suppose that for the parameters of the $q$-Hahn TASEP, \eqref{munucond} holds.
Then the contour $\D_\theta$ is of steep descent for the function $\Re(f_0)$
in the sense that the function attains its maximum at $q^\theta$ corresponding to $t=0$, it increases for $t\in(-\pi,0)$ and it decreases for $t\in(0,\pi)$.
\end{proposition}

\begin{remark}\label{rem:munucond}
The condition $q\le\nu$ is needed for the use of the first part of Lemma~\ref{lemma:compareg} to obtain \eqref{newsum2} which is an ingredient to the steep descent property in Proposition~\ref{prop:steep1}.
This condition is also used in the proof of Proposition~\ref{prop:steep2} for the application of \eqref{cauchy} with $\alpha=\nu q^k$ and $\beta=q^{k+1}$.
The condition $\mu\le1/2$ is imposed in Proposition~\ref{prop:steep2}, because the inequality \eqref{cauchy} could be proved for $\beta\le1/2$.
The $1/2$ seems numerically to be close to optimal.
The condition $\mu\le1/2$ could be weakened in Proposition~\ref{prop:steep1}, because $(1-q^\theta)\mu/(1-\mu)\le1$ is enough to obtain \eqref{newsum1} from the first part of Lemma~\ref{lemma:compareg}.
The latter is a weaker condition, but $\mu\le1/2$ has to be assumed for our proof of Proposition~\ref{prop:steep2} to work.
\end{remark}

The function given by
\[g(b,s)=\frac{b\sin s}{1+b^2-2b\cos s}\]
is useful for the proof of the propositions about steep descent contours.
It has the following properties.

\begin{lemma}\label{lemma:compareg}
\begin{enumerate}
\item
If $-1<b\le c<1$ and $0\le s\le\pi$, then
\begin{equation}\label{compareg}
\frac{(1-b)^2}b g(b,s)\ge \frac{(1-c)^2}c g(c,s)
\end{equation}
and the inequality above is sharp for $s\in(0,\pi)$ if $b<c$.

\item
For $0<b<1$ and $0\le t\le\pi$, one has
\[\left(\frac{(1-b/2)^2}{b/2} g(b/2,t)\right)^2\ge\frac{(1-b)^2}b g(b,t)\sin t.\]
with strict inequality for $t\in(0,\pi)$.
\end{enumerate}
\end{lemma}

\begin{remark}
The first part of Lemma~\ref{lemma:compareg} is designed to compare the different terms in the derivative of $f_0$ along a circular contour given in \eqref{deriv-f0}.
The inequality \eqref{compareg} is sharp in the sense that the two sides as functions of $s$ are tangential at $s=0$.
The factor $1/2$ in the second part of the lemma seems numerically to be close to optimal.
\end{remark}

\begin{proof}[Proof of Lemma~\ref{lemma:compareg}]
\begin{enumerate}
\item
The inequality is the consequence of the fact that
\[\frac{\partial}{\partial b}\left(\frac{(1-b)^2}b g(b,s)\right)=-\frac{2(1-b^2)(1-\cos s)\sin s}{(1+b^2-2b\cos s)^2}\]
is non-positive for $-1<b<1$ and $0\le s\le\pi$ and it is strictly negative for $-1<b<1$ and $0<s<\pi$.

\item
Calculations show that
\begin{multline*}
\left(\frac{(1-b/2)^2}{b/2} g(b/2,t)\right)^2-\frac{(1-b)^2}b g(b,t)\sin t\\
=\frac{4b^2(4b-3b^3+8(1-b)^2(1+\cos t))\sin^2(t/2)\sin^2 t}{(4+b^2-4b\cos t)^2(1+b^2-2b\cos t)}
\end{multline*}
which is non-negative for the given parameter values and strictly positive for for $t\in(0,\pi)$.
\end{enumerate}
\end{proof}

The series representations of the $q$-digamma function and that of its derivative are used in the proof below.
They are expressed as
\begin{align}
\Psi_q(Z)&=-\log(1-q)+\log q\sum_{k=0}^\infty \frac{q^{Z+k}}{1-q^{Z+k}},\label{Psiseries}\\
\Psi_q'(Z)&=(\log q)^2\sum_{k=0}^\infty \frac{q^{Z+k}}{(1-q^{Z+k})^2}.\label{Psi'series}
\end{align}

\begin{proof}[Proof of Proposition~\ref{prop:steep1}]
We investigate the function $-\Re(f_0)$ along $\C_\theta$, i.e.\ $w(s)=1-re^{\I s}$ for $s\in[0,\pi]$ where we use the notation $r=1-q^\theta$ as we will do throughout this proof.
The case $s\in(-\pi,0]$ is similar.
One can check by calculation that
\begin{equation}\label{derivlogw}
\Re\frac{\d}{\d s}\log w(s)=g(r,s)
\end{equation}
and that for any $0<\alpha\le1$,
\begin{equation}\label{derivlogPoch}
\Re\frac{\d}{\d s}\log(\alpha w(s);q)_\infty=\sum_{k=0}^\infty g\left(-r\frac{\alpha q^k}{1-\alpha q^k},s\right)
\end{equation}
where the $k=0$ term for $\alpha=1$ is understood as the limit $\lim_{b\to-\infty}g(b,s)=0$.
Using \eqref{derivlogw}--\eqref{derivlogPoch} for \eqref{deff0}, we have
\begin{equation}\label{deriv-f0}\begin{aligned}
-\Re\frac{\d}{\d s} f_0(w(s))=fg(r,s)&+\kappa\sum_{k=0}^\infty\left(g\left(-r\frac{\mu q^k}{1-\mu q^k},s\right)-g\left(-r\frac{\nu q^k}{1-\nu q^k},s\right)\right)\\
&+\sum_{k=0}^\infty\left(g\left(-r\frac{\nu q^k}{1-\nu q^k},s\right)-g\left(-r\frac{q^k}{1-q^k},s\right)\right).
\end{aligned}\end{equation}
Note that the $k=0$ term of the last summand is $0$ by the observation above.

The first part of Lemma~\ref{lemma:compareg} for $b=-r\mu q^k/(1-\mu q^k)$ and $c=-r\nu q^k/(1-\nu q^k)$ yields that
\begin{equation}\label{newsum1}
\sum_{k=0}^\infty\left(\frac{\left(1+r\frac{\mu q^k}{1-\mu q^k}\right)^2}{r\frac{\mu q^k}{1-\mu q^k}}\frac{r\frac{\nu q^k}{1-\nu q^k}}{\left(1+r\frac{\nu q^k}{1-\nu q^k}\right)^2}
g\left(-r\frac{\mu q^k}{1-\mu q^k},s\right)-g\left(-r\frac{\nu q^k}{1-\nu q^k},s\right)\right)\le0
\end{equation}
and the inequality is strict for $s\in(0,\pi)$.
Similarly for $b=-r\nu q^k/(1-\nu q^k)$ and $c=-rq^{k+1}/(1-q^{k+1})$, one has
\begin{equation}\label{newsum2}
\sum_{k=0}^\infty\left(\frac{\left(1+r\frac{\nu q^k}{1-\nu q^k}\right)^2}{r\frac{\nu q^k}{1-\nu q^k}}\frac{r\frac{q^{k+1}}{1-q^{k+1}}}{\left(1+r\frac{q^{k+1}}{1-q^{k+1}}\right)^2}
g\left(-r\frac{\nu q^k}{1-\nu q^k},s\right)-g\left(-r\frac{q^{k+1}}{1-q^{k+1}},s\right)\right)\le0.
\end{equation}
Note that since the function $b\mapsto(1+b)^2/b$ is decreasing on $b\in(0,1)$,
the prefactor of $g(-r\mu q^k/(1-\mu q^k),s)$ in \eqref{newsum1} and that of $g(-r\nu q^k/(1-\nu q^k),s)$ in \eqref{newsum2} are smaller than $1$.

Hence in order to complete the argument, one has to compare the remainder of the terms $g(-r\mu q^k/(1-\mu q^k),s)$ and $g(-r\nu q^k/(1-\nu q^k),s)$ in \eqref{deriv-f0} with $g(r,s)$.
To this end, we use the first part of Lemma~\ref{lemma:compareg} for $b=-r\mu q^k/(1-\mu q^k)$ and $c=r$ which gives us that
\begin{equation}\label{rest1}
\frac{(1-r)^2}r\frac{r\frac{\mu q^k}{1-\mu q^k}}{\left(1+r\frac{\mu q^k}{1-\mu q^k}\right)^2}g(r,s)+g\left(-r\frac{\mu q^k}{1-\mu q^k},s\right)\le0
\end{equation}
and for $b=-r\nu q^k/(1-\nu q^k)$ and $c=r$, we have
\begin{equation}\label{rest2}
\frac{(1-r)^2}r\frac{r\frac{\nu q^k}{1-\nu q^k}}{\left(1+r\frac{\nu q^k}{1-\nu q^k}\right)^2}g(r,s)+g\left(-r\frac{\nu q^k}{1-\nu q^k},s\right)\le0.
\end{equation}

What remains to show is that
\begin{equation}\label{f=}\begin{aligned}
f&=\kappa\sum_{k=0}^\infty\left(1-\frac{\left(1+r\frac{\mu q^k}{1-\mu q^k}\right)^2}{r\frac{\mu q^k}{1-\mu q^k}}\frac{r\frac{\nu q^k}{1-\nu q^k}}{\left(1+r\frac{\nu q^k}{1-\nu q^k}\right)^2}\right)
\frac{(1-r)^2}r\frac{r\frac{\mu q^k}{1-\mu q^k}}{\left(1+r\frac{\mu q^k}{1-\mu q^k}\right)^2}\\
&\qquad+\sum_{k=0}^\infty\left(1-\frac{\left(1+r\frac{\nu q^k}{1-\nu q^k}\right)^2}{r\frac{\nu q^k}{1-\nu q^k}}\frac{r\frac{q^{k+1}}{1-q^{k+1}}}{\left(1+r\frac{q^{k+1}}{1-q^{k+1}}\right)^2}\right)
\frac{(1-r)^2}r\frac{r\frac{\nu q^k}{1-\nu q^k}}{\left(1+r\frac{\nu q^k}{1-\nu q^k}\right)^2}
\end{aligned}\end{equation}
because of the following.
Let us multiply \eqref{rest1} by $\kappa$ times the factor between parentheses in the first sum of \eqref{f=}
and multiply \eqref{rest2} by the factor between parentheses in the second sum of \eqref{f=} and sum these up for $k$.
Then add $\kappa$ times \eqref{newsum1} and \eqref{newsum2} to the sum.
This altogether is to be compared to \eqref{deriv-f0}.
Note that the coefficients of $g(-r\mu q^k/(1-\mu q^k),s)$, $g(-r\nu q^k/(1-\nu q^k),s)$ and $g(-rq^k/(1-q^k),s)$ coincide.
(Remember that the last term for $k=0$ in \eqref{deriv-f0} is $0$.)
On the other hand, the coefficients of $g(r,s)$ are exactly the two sides of \eqref{f=},
therefore if \eqref{f=} holds true, then the derivative \eqref{deriv-f0} is non-positive for $s\in[0,\pi]$ and negative for $s\in(0,\pi)$.

By multiplication in \eqref{f=}, for $\alpha\in(0,1]$, we get terms of the form
\begin{equation}\label{alphasum}\begin{aligned}
\sum_{k=0}^\infty \frac{(1-r)^2\frac{\alpha q^k}{1-\alpha q^k}}{\left(1+r\frac{\alpha q^k}{1-\alpha q^k}\right)^2}
&=-(1-q^\theta)\sum_{k=0}^\infty \frac{\alpha q^{\theta+k}}{(1-\alpha q^{\theta+k})^2}+\sum_{k=0}^\infty \frac{\alpha q^{\theta+k}}{1-\alpha q^{\theta+k}}\\
&=-(1-q^\theta)\frac{\Psi'_q(\theta+\log_q\alpha)}{(\log q)^2}+\frac{\Psi_q(\theta+\log_q\alpha)+\log(1-q)}{\log q}
\end{aligned}\end{equation}
where we used that $r=1-q^\theta$ and the series expansions \eqref{Psiseries}--\eqref{Psi'series}.
The right-hand side of \eqref{f=} equals $\kappa$ times the difference of \eqref{alphasum} for $\alpha=\mu$ and for $\alpha=\nu$ and the difference of \eqref{alphasum} for $\alpha=\nu$ and for $\alpha=q$.
Note that the sum in \eqref{alphasum} for $\alpha=q$ can be replaced by the one for $\alpha=1$, because the extra term is $0$ as it can be seen from the second expression in \eqref{alphasum}.
Hence the right-hand side of \eqref{f=} can be written as
\begin{multline*}
\frac1{\log q}\left(\kappa\left(\Psi_q(\theta+\log_q\mu)-\Psi_q(\theta+\log_q\nu)\right)+\Psi_q(\theta+\log_q\nu)-\Psi_q(\theta)\right)\\
-\frac{1-q^\theta}{(\log q)^2}\left(\kappa\left(\Psi'_q(\theta+\log_q\mu)-\Psi'_q(\theta+\log_q\nu)\right)+\Psi'_q(\theta+\log_q\nu)-\Psi'_q(\theta)\right)
\end{multline*}
which is exactly $f$ by \eqref{defkappa}--\eqref{deff} as required.
\end{proof}

\begin{proof}[Proof of Proposition~\ref{prop:steep2}]
First note that
\[\Re\frac{\d}{\d t}\log z(t)=0\]
and that for $0<\alpha\le1$,
\[\Re\frac{\d}{\d t}\log(\alpha z(t);q)_\infty=\sum_{k=0}^\infty g(\alpha q^{\theta+k},t).\]
We write the derivative of $\Re(f_0)$ as
\begin{equation}\label{derivf0}
\Re\frac{\d}{\d t} f_0(z(t))
=g(q^\theta,t)-\sum_{k=0}^\infty \left[\kappa\left(g(\mu q^{\theta+k},t)-g(\nu q^{\theta+k},t)\right)+g(\nu q^{\theta+k},t)-g(q^{\theta+k+1},t)\right].
\end{equation}
We prove that it is non-positive for $t\in[0,\pi]$ with strict inequality for $t\in(0,\pi)$.
The corresponding inequality for $t\in[-\pi,0]$ is similar.

As a consequence of \eqref{defkappa} and \eqref{Psi'series}, one can write the coefficient of $g(q^\theta,t)$ in the derivative above as
\begin{equation}\label{1=}
1=\frac{(1-q^\theta)^2}{q^\theta}\sum_{k=0}^\infty \left(\kappa\left(\frac{\mu q^{\theta+k}}{(1-\mu q^{\theta+k})^2}-\frac{\nu q^{\theta+k}}{(1-\nu q^{\theta+k})^2}\right)
+\frac{\nu q^{\theta+k}}{(1-\nu q^{\theta+k})^2}-\frac{q^{\theta+k+1}}{(1-q^{\theta+k+1})^2}\right).
\end{equation}
We shall prove that if $0<\alpha\le\beta\le1/2$, then
\begin{equation}\label{cauchy}
\left(\frac{\beta q^\theta}{(1-\beta q^\theta)^2}-\frac{\alpha q^\theta}{(1-\alpha q^\theta)^2}\right)\frac{(1-q^\theta)^2}{q^\theta} g(q^\theta,t)\le g(\beta q^\theta,t)-g(\alpha q^\theta,t)
\end{equation}
with strict inequality for $t\in(0,\pi)$ and $\alpha<\beta$.
An application of \eqref{cauchy} for $\alpha=\nu q^k$ and $\beta=\mu q^k$ and an application for $\alpha=\nu q^k$ and $\beta=q^{k+1}$ together with \eqref{1=}
proves that \eqref{derivf0} is non-positive.

To show \eqref{cauchy}, we write it as
\begin{equation}\label{cauchy2}
\frac{(1-q^\theta)^2}{q^\theta} g(q^\theta,t)\le \frac{g(\beta q^\theta,t)-g(\alpha q^\theta,t)}{\frac{\beta q^\theta}{(1-\beta q^\theta)^2}-\frac{\alpha q^\theta}{(1-\alpha q^\theta)^2}}
\end{equation}
and we observe that by Cauchy's mean value theorem the right-hand side of \eqref{cauchy2} can be written as
\begin{equation}\label{derivratio}
\frac{\d}{\d b}g(b,t)\left(\frac{\d}{\d b} \frac b{(1-b)^2}\right)^{-1}=\left(\frac{(1-b)^2}b g(b,t)\right)^2\frac1{\sin t}
\end{equation}
for some $b\in(\alpha q^\theta,\beta q^\theta)$.
By the first part of Lemma~\ref{lemma:compareg}, $b\mapsto(1-b)^2 g(b,t)/b$ is decreasing.
Since $\beta\le1/2$, \eqref{derivratio} for $b\in(\alpha q^\theta,\beta q^\theta)$ cannot be smaller than its value at $q^\theta/2$, i.e.
\[\left(\frac{(1-q^\theta/2)^2}{q^\theta/2}g(q^\theta/2,t)\right)^2\frac1{\sin t}.\]
Applying the second part of Lemma~\ref{lemma:compareg} gives the inequality \eqref{cauchy2} which completes the proof.
\end{proof}

\section{Proofs of propositions}\label{s:prop}

This section contains the proofs of the propositions stated in Section~\ref{s:analysis} which lead to Theorem~\ref{thm:Fredholmconv}.
For later use, note that differentiation of \eqref{deff0}--\eqref{deff2} gives the following Taylor series expansions
\begin{align}
f_0(q^Z)&=f_0(q^\theta)+\frac\chi3(Z-\theta)^3+\O((Z-\theta)^4),\label{f0Taylor}\\
f_1(q^Z)&=f_1(q^\theta)+\frac{c\phi'}2(Z-\theta)^2+\O((Z-\theta)^3),\label{f1Taylor}\\
f_2(q^Z)&=f_2(q^\theta)+\beta_x(Z-\theta).\label{f2Taylor}
\end{align}

\begin{proof}[Proof of Proposition~\ref{prop:localization}]
The proof consists of the following three steps.
We first deform the integration contour for the kernel $K_x$ to the steep descent contour.
Then we show that Proposition~\ref{prop:localization} holds with $\varphi=\pi/2$ instead of $\varphi\in(0,\pi/2)$.
In the last step, we deform the short contours so that we get the statement for $\varphi\in(0,\pi/2)$.

\emph{Step 1: Contour deformation.}
The first observation is that as long as conditions \eqref{munucond}--\eqref{thetacond} hold, then
\[\det(\id+K_x)_{L^2(C_0)}=\det(\id+K_x)_{L^2(\wt\C_\theta)},\]
i.e.\ the contour $C_0$ can be blowed up to $\wt\C_\theta$.
This simply follows from the Cauchy theorem since the singularities coming from $f_0$ at $-\log_q \mu q^k$, $-\log_q \nu q^k$ and $\log_q q^k$ for $k=0,1,2,\dots$
are by condition \eqref{munucond} all smaller than $\log_q2$ and the point of $\wt\C_\theta$ with the smallest real part is $\log_q(2-q^\theta)$.
One the other hand, the condition \eqref{thetacond} ensures that no pole coming from the sine in the denominator is crossed along the deformation,
since the real part of the points of $\wt\C_\theta$ is between $\log_q(2-q^\theta)$ and $\theta$ and the difference of the two is assumed to be less than $1$ by \eqref{thetacond}.

\emph{Step 2: Localization to short contours.}
In this step, we prove the statement of the proposition with $\varphi=\pi/2$ instead of $\varphi\in(0,\pi/2)$.
Recall from Definition~\ref{def:contours} that $s\mapsto w(s)$ parametrized the contour $\C_\theta$ and hence $s\mapsto\log_q w(s)$ parametrized the contour $\wt\C_\theta$.
It allows for writing the Fredholm determinant as
\begin{multline}\label{Fredholmseries}
\det(\id+K_x)_{L^2(\wt\C_\theta)}\\
=\sum_{k=0}^\infty \frac1{k!} \int_{-\pi}^\pi\d s_1\dots\int_{-\pi}^\pi\d s_k \det\left(K_x(\log_q w(s_i),\log_q w(s_j))\frac{\d}{\d s_i}\log_q w(s_i)\right)_{i,j=1}^k.
\end{multline}

The kernel in \eqref{Fredholmseries} diverges logarithmically in the neighbourhood of $s_i=0$, but it is bounded otherwise.
More precisely, there is a constant $C$ such that
\[|K_x(\log_q w(s),\log_q w(s'))|\le C(1+(\log|s|)_-+(\log|s'|)_-).\]
On the other hand, we use the fact that the contour $\wt\C_\theta$ is of steep descent for the function $W\mapsto-\Re(f_0(q^W))$ as an immediate consequence of Proposition~\ref{prop:steep1}.
Since this function gives the main contribution in the exponent in the $W$ variable, the kernel converges to $0$ exponentially as $N\to\infty$
for all $W\in\wt\C_\theta$ except for a $\delta$-neighbourhood of $\theta$.
Hence by dominated convergence, the integral along the contour $\wt\C_\theta$ in the Fredholm series of $K_x$ in \eqref{Fredholmseries} can be neglected apart from a $\delta$-neighbourhood of $\theta$
by making an error of order $\O(\exp(-c\delta^3 N))$ for some $c>0$.
Keeping the endpoints of the remaining contour, it can be replaced by $V_{\theta,\pi-\varphi}^\delta$ for some $\varphi\in(0,\pi/2)$ by Cauchy's theorem.
Note that in the last step, the orientation of the contour changes.

With a similar argument, we can localize the $Z$-contour as well.
By linearity, one can take out the $Z$ integrations from the determinant in \eqref{Fredholmseries} to obtain a sum where the $k$th term is a $2k$-fold integration.
It is still integrable, since the behaviour in the $Z$ variables is $e^{-\pi\Im(Z)}$ due to the sine in the denominator.
The function $\Re(f_0(q^Z))$ is periodic along the contour $\theta+\I\R$ with period $2\pi/|\log q|$ in the imaginary direction.
The contour $\{\theta+\I t:t\in[\pi/\log q,-\pi/\log q]\}$ is however of steep descent for $\Re(f_0(q^Z))$ by Proposition~\ref{prop:steep2}.
Therefore, the steep descent property and the periodicity implies that by making an exponentially small error in $N$, we can restrict the $Z$ integral to the set $\cup_{k\in\Z}I_k$ where
$I_k=\{\theta+\I t:t-2k\pi/\log q\in[-\delta,\delta]\}$, in particular, $I_0=V_{\theta,\pi/2}^\delta$.

Now we argue that in the limit, only the integral over $I_0$ survives.
Let us consider the change of variables
\[W=\theta+wN^{-1/3},\qquad Z=\theta+ik\frac{2\pi}{\log q}+zN^{-1/3}.\]
Now the term $N^{-1/3}/\sin(\pi(W-Z))$ for $k=0$ converges to $1/(\pi(w-z))$, whereas for $k\neq0$, we have
\[\frac{N^{-1/3}}{\sin(\pi(W-Z))}\asymp N^{-1/3}e^{-\pi|\Im Z|}\asymp N^{-1/3}e^{-\frac{2\pi^2}{|\log q|}|k|}\]
which is a summable but it is of smaller order than the term for $k=0$.
It means that the integral over $\cup_{k\in\Z}I_k$ can be replaced by the one over $I_0$ in the limit.
This proves the proposition for $\varphi=\pi/2$.

\emph{Step 3: Deformation of short contours.}
What we show in this step is that the contour for the $Z$ integral in the kernel $K_{x,\delta}$ that can be taken to be a segment from $\theta-\I\delta$ to $\theta+\I\delta$
can be replaced by $V_{\theta,\varphi}^\delta$ by possibly choosing a smaller $\delta$.
First with the $\varphi$ obtained in the localization of the $W$ contour and by using Cauchy's theorem, we replace the integration path for $Z$ by the union of
\begin{align*}
S_1&=[\theta-\I\delta,\theta+\delta\cos\varphi-\I\delta\sin\varphi],&
S_2&=[\theta+\delta\cos\varphi-\I\delta\sin\varphi,\theta],\\
S_3&=[\theta,\theta+\delta\cos\varphi+\I\delta\sin\varphi],&
S_4&=[\theta+\delta\cos\varphi+\I\delta\sin\varphi,\theta+\I\delta]
\end{align*}
where we mean four segments in the complex plane with the given endpoints.

The new integration path is not a steep descent contour any more, nevertheless one can proceed as follows.
The function $Z\mapsto f_0(q^Z)$ behaves around $\theta$ as $f(q^\theta)+(Z-\theta)^3$ in the leading order by the Taylor expansion \eqref{f0Taylor}.
On the other hand by Proposition~\ref{prop:steep2}, we know that the value of $\Re(f_0(q^Z))$ as $Z\in\theta+\I\R$ is smaller than $f_0(q^\theta)$.
Hence by taking $\delta$ sufficiently small so that the Taylor approximation works well enough, we can achieve that $\Re(f_0(q^Z))$ as $Z\in S_1\cup S_4$ is strictly smaller than $f_0(q^\theta)$.
Similarly to the second step of this proof, we can further reduce the integration path to $S_2\cup S_3=V_{\theta,\varphi}^\delta$ by making an exponentially small error in $N$.
This completes the proof.
\end{proof}

\begin{proof}[Proof of Proposition~\ref{prop:kernelconv}]
We apply the change of variables
\begin{equation}\label{localchov}
W=\theta+wN^{-1/3},\quad W'=\theta+w'N^{-1/3},\quad Z=\theta+zN^{-1/3}
\end{equation}
as in \eqref{defKxdelta} and use the Taylor expansions \eqref{f0Taylor}--\eqref{f2Taylor}.
It gives that up to an error $\O(N^{-1/3})$ in the exponent, the kernel $K_{x,\delta}^N(w,w')$ is close to $K_{x,\delta N^{1/3}}'(w,w')$
for any $w,w'\in V_{0,\pi-\varphi}^{\delta N^{1/3}}$ as $N\to\infty$.
By using the inequality $|e^x-1|\le|x|e^{|x|}$, we obtain that the kernel with and without the error in the exponent differ by an $\O(N^{-1/3})$ term.

In order to get that the Fredholm determinants are also close, we need a uniform fast decaying bound on $K_{x,\delta}^N(w,w')$.
The main term in the exponent is
\[-N\Re f_0(q^W)=-\chi\Re\frac{w^3}3+\O_\epsilon({N^{-1/3}}w^4)=-\chi\Re\frac{w^3}3+\O_\epsilon(w^3)\]
for any $\epsilon>0$ where $\O_\epsilon(w^3)$ means that the error term is at most $\epsilon w^3$.
By taking $\delta$ small enough, $\epsilon$ can be arbitrarily small, hence the error term is negligible compared to the cubic behaviour of $-\chi\Re w^3/3$.
The error terms coming from $f_1$ and $f_2$ in the exponent are similarly dominated.
Hence the difference of the Fredholm determinants goes to $0$ as $N\to\infty$ by dominated convergence, which proves the proposition.
\end{proof}

\begin{proof}[Proof of Proposition~\ref{prop:kernelextend}]
Since the integrand in \eqref{kernelK'} has cubic exponential decay in $w$ and $z$ along the given contours $V_{0,\pi-\varphi}^\infty$ and $V_{0,\varphi}^\infty$ respectively,
the convergence of the Fredholm determinants follows by dominated convergence similarly to the earlier proofs.
\end{proof}

\begin{proof}[Proof of Proposition~\ref{prop:rewritekernel}]
By definition \eqref{kernelK'}, we can write the kernel $K_{x,\infty}'$ on $L^2(V_{0,\pi-\varphi}^\infty)$ as
\begin{equation}\label{kernelAB}
K_{x,\infty}'(w,w')=(AB)(w,w')
\end{equation}
where
\begin{align*}
A(w,\lambda)&=e^{-\chi w^3/3+c\phi'w^2/2-\beta_x w+\lambda w}\\
B(\lambda,w')&=\frac1{2\pi\I}\int_{V_{0,\varphi}^\infty}\frac{\d z}{z-w'}e^{\chi z^3/3-c\phi'z^2/2+\beta_x z-\lambda z}
\end{align*}
with $\lambda\in\R_+$ and the composition $AB$ on the right-hand side of \eqref{kernelAB} is also meant in $L^2(\R_+)$.
The equality \eqref{kernelAB} can be seen since
\[\frac1{z-w}=\int_0^\infty\d\lambda\,e^{-\lambda(z-w)}\]
as long as $\Re(z-w)>0$.

Hence we can write
\[\det(\id-K_{x,\infty}')_{L^2(V_{0,\pi-\varphi}^\infty)}=\det(\id-AB)_{L^2(V_{0,\pi-\varphi}^\infty)}=\det(\id-BA)_{L^2(\R_+)}\]
with
\[(BA)(y,y')=\int_0^\infty\d\lambda\frac1{(2\pi\I)^2}\int_{V_{0,\pi-\varphi}^\infty}\d w\int_{V_{0,\varphi}^\infty}\d z
\frac{e^{\chi z^3/3-c\phi'z^2/2+(\beta_x-y-\lambda)z}}{e^{\chi w^3/3-c\phi'w^2/2+(\beta_x-y'-\lambda)w}}.\]
Using the general formula
\[\frac1{2\pi\I}\int_{V_{0,\varphi}^\infty}\exp\left(a\frac{z^3}3+bz^2+cz\right)\d z=a^{-1/3}\exp\left(\frac{2b^3}{3a^2}-\frac{bc}a\right)\Ai\left(\frac{b^2}{a^{4/3}}-\frac c{a^{1/3}}\right),\]
we get that
\[(BA)(y,y')=\chi^{-1/3}e^{-c\phi'(y-y')/(2\chi)}\int_0^\infty\d\lambda\Ai\left(\frac y{\chi^{1/3}}+x+\lambda\right)\Ai\left(\frac{y'}{\chi^{1/3}}+x+\lambda\right)\]
by \eqref{defbeta} and after the change of variable $\lambda\to\chi^{1/3}\lambda$.
After conjugation by the exponential prefactor and by rescaling the Fredholm determinant, we get that
\[\det(\id-BA)_{L^2(\R_+)}=\det(\id-K_{\Ai,x})_{L^2(\R_+)}=F_{\rm GUE}(x)\]
as required.
\end{proof}

\section*{Acknowledgements}
The author thanks Guillaume Barraquand and Ivan Corwin for stimulating discussions and comments related to the present work, and an anonymous referee for a detailed review.
This research was supported by the European Union and the State of Hungary, co-financed by the European Social Fund in the framework of T\'AMOP 4.2.4.\ A/1--11--1--2012--0001 National Excellence Program.
The author is grateful for the Postdoctoral Fellowship of the Hungarian Academy of Sciences and for the Bolyai Research Scholarship.
His work was partially supported by OTKA (Hungarian National Research Fund) grant K100473.

\end{document}